\documentclass[preprint,12pt]{elsarticle}
\usepackage{amsthm,amsfonts,amssymb,amscd,amsmath,enumerate,verbatim,calc,graphicx,geometry}
\usepackage[all]{xy}
\newtheorem{theorem}{Theorem}[section]
\newtheorem{lemma}[theorem]{Lemma}

\newtheorem{corollary}[theorem]{Corollary}
\theoremstyle{definition}
\theoremstyle{definitions}

\newtheorem{remark}[theorem]{Remark}
\newtheorem{example}[theorem]{Example}
\theoremstyle{notations}

\theoremstyle{remarks}

\journal{}

\begin{document}

\begin{frontmatter}



\title{On Topological Shape Homotopy Groups}


\author[]{Tayyebe~Nasri}
\ead{tnasri72@yahoo.com}
\author[]{Fateme~Ghanei}
\ead{fatemeh.ghanei91@gmail.com}
\author[]{Behrooz~Mashayekhy\corref{cor1}}
\ead{bmashf@um.ac.ir}
\author[]{Hanieh~Mirebrahimi}
\ead{h\_mirebrahimi@um.ac.ir}

\address{Department of Pure Mathematics, Center of Excellence in Analysis on Algebraic Structures, Ferdowsi University of Mashhad,\\
P.O.Box 1159-91775, Mashhad, Iran.}
\cortext[cor1]{Corresponding author}

\begin{abstract}
In this paper, using the topology on the set of shape morphisms between arbitrary topological spaces $X$,
$Y$, $Sh(X,Y)$, defined by Cuchillo-Ibanez et al. in 1999, we consider a topology on the shape homotopy groups of arbitrary topological spaces which make them Hausdorff topological groups. We then exhibit an example in which $\check{\pi}_k^{top}$ succeeds in distinguishing the shape type of $X$ and $Y$ while $\check{\pi}_k$ fails, for all $k\in \Bbb{N}$. Moreover, we present some basic properties of topological shape homotopy groups, among them commutativity of $\check{\pi}_k^{top}$ with finite product of compact Hausdorff spaces. Finally, we consider a quotient topology on the $k$th shape group induced by the $k$th shape loop space and show that it coincides with the above topology.
\end{abstract}


\end{frontmatter}

\section{Introduction and Motivation}

Moron et al. \cite{Mo1} gave a complete, non-Archimedean metric (or ultrametric) on the set of shape morphisms between
two unpointed compacta (compact metric spaces) $X,$ $Y$ , written $Sh(X, Y)$. Then they mentioned that this construction can be translated to the pointed case. Consequently, as a particular case, they obtained a complete ultrametric that induces a norm on the shape groups of a compactum $Y$ and then presented some results on these topological groups \cite{Mo2}. Also Cuchillo-Ibanez et al. \cite{CM2} constructed many generalized ultrametrics in the set of shape morphisms between topological spaces and obtained semivaluations and valuations on the groups of shape equivalences and $k$th shape groups. On the other hands, Cuchillo-Ibanez et al. \cite{CM}  introduced a topology on the set $Sh(X, Y)$, where $X$ and $Y$ are arbitrary topological spaces, in such a way that it extended topologically the construction given in \cite{Mo1}. Also, Moszynska \cite{M} showed that the $k$th shape group $\check{\pi}_k(X,x)$, $k\in \Bbb{N}$, is isomorphic to the set $Sh((S^k, *),(X x))$ consists of all shape morphisms $(S^k, *)\rightarrow(X, x)$ with a group operation. In this paper we consider the latter topology on the set of shape morphisms between pointed spaces, consequently we obtain a topology on the shape groups of arbitrary spaces. We show that the $k$th shape group $\check{\pi}_k(X,x)$, $k\in \Bbb{N}$, with the above topology is a Hausdorff topological group, denoted by $\check{\pi}_k^{top}(X,x)$. We then exhibit an example in which $\check{\pi}_k^{top}$ succeeds in distinguishing the shape type of $X$ and $Y$ while $\check{\pi}_k$ fails, for all $k\in \Bbb{N}$. In fact, we exhibit two topological spaces $X$ and $Y$ in which whose $k$th shape groups are isomorphic while whose topological $k$th shape groups are not isomorphic, for all $k\in\Bbb{N}$. Also we show that $\check{\pi}_k^{top}$ preserves the product $X\times Y$ provided $X$ and $Y$ are compact Hausdorff spaces.  Moreover, we obtain some topological properties of these groups and we study some properties in shape as movability. We show that movability can be transferred from a pointed topological spaces $(X,x)$ to $\check{\pi}_k^{top}(X,x)$ under some conditions. Also we show that if $X\in HPol$, then $\check{\pi}_k^{top}(X,x)$ is discrete. Moreover, if $\mathbf{p}:X\longrightarrow\mathbf{X}=(X_{\lambda}, p_{\lambda\lambda'},\Lambda) $ is an HPol-expansion of $X$ and $\check{\pi}_k^{top}(X,x)$ is discrete, then $\check{\pi}_k^{top}(X,x)\leq\check{\pi}_k^{top}(X_{\lambda},x_{\lambda})$, for some $\lambda\in \Lambda$ and for all $k\in \Bbb{N}$.

Endowed with the quotient topology induced by the natural surjective map $q:\Omega^k(X,x)\rightarrow \pi_k(X,x)$, where $\Omega^k(X,x)$ is the $k$th loop space of $(X,x)$ with the compact-open topology, the familiar homotopy group $\pi_k(X,x)$ becomes a quasitopological group which is called the quasitopological $k$th homotopy group of the pointed space $(X,x)$, denoted by $\pi_k^{qtop}(X,x)$ (see \cite{B0,Br,Bra,G1}).

Biss \cite{B0} proved that $\pi_1^{qtop}(X,x)$ is a topological group. However, Calcut and McCarthy \cite{CM1} and Fabel \cite{F2} showed that there is a gap in the proof of \cite[Proposition 3.1]{B0}. The misstep in the proof is repeated by Ghane et al. \cite{G1} to prove that $\pi_k^{qtop}(X,x)$ is a topological group \cite[Theorem 2.1]{G1} (see also \cite{CM1}).

Calcut and McCarthy \cite{CM1} showed that $\pi_1^{qtop}(X,x)$ is a homogeneous space and more precisely, Brazas \cite{Br} mentioned that $\pi_k^{qtop}(X,x)$ is a quasitopological group in the sense of \cite{A}.

Calcut and McCarthy \cite{CM1} proved that for a path connected and locally path connected space $X$,  $\pi_1^{qtop}(X)$ is a discrete topological group if and only if $X$ is semilocally 1-connected (see also \cite{Br}). Pakdaman et al. \cite{P1} showed that for a locally $(k-1)$-connected space $X$, $\pi_k^{qtop}(X,x)$ is discrete if and only if $X$ is semilocally n-connected at $x$ (see also \cite{G1}).
Fabel \cite{F2,F3} and Brazas \cite{Br} presented some spaces for which their quasitopological homotopy groups are not topological groups.
 Moreover, despite of Fabel's result \cite{F2} that says the quasitopological fundamental group of the Hawaiian earring is not a topological group, Ghane et al. \cite{G2} proved that the topological $k$th homotopy group of an $k$-Hawaiian like space is a prodiscrete metrizable topological group, for all $k\geq 2$.


For an HPol-expansion $\mathbf{p}:X\rightarrow (X_{\lambda},p_{\lambda\lambda'},\Lambda)$ of $X$, Brazas \cite{Br} introduced a topology on $\check{\pi}_k(X)$. Here we denote it by $\check{\pi}_k^{inv}(X)$. He mentioned that since $\pi_k^{qtop}(X_{\lambda})$ is discrete, for all $k\in \Bbb{N}$, one can define the $k$th topological shape group of $X$ as the limit $\check{\pi}^{inv}_k(X)=\displaystyle{\lim_{\leftarrow}\pi_k^{qtop}(X_{\lambda})}$ which is an inverse limit of discrete groups and so it is a Hausdorff topological group  \cite[Remark 2.19]{Br}.  In this paper, we prove that $\check{\pi}_k^{inv}(X,x)\cong\check{\pi}_k^{top}(X,x)$ as topological groups for every pointed topological space $(X,x)$. The second view at this topology implies some results which seem are not obtained with the first view.

Finally, for an HPol$_*$-expansion $\mathbf{p}:(X,x)\rightarrow ((X_{\lambda},x_{\lambda}),p_{\lambda\lambda'},\Lambda)$ of a pointed topological space $(X,x)$, we define  the $k$th shape loop space as a subspace of $\prod_{\lambda\in \Lambda} \Omega^{k}(X_{\lambda},x_{\lambda})$  and consider a quotient topology on the $k$th shape group induced by the $k$th shape loop space. Then we show that this quotient topology on the $k$th shape group coincides with the topology $\check{\pi}_k^{top}$. 
\section{Preliminaries}
In this section, we recall some of the main notions concerning the shape category and the pro-HTop (see \cite{MS}).
Let $\mathbf{X}=(X_{\lambda},p_{\lambda\lambda'},\Lambda)$ and $\mathbf{Y}=(Y_{\mu},q_{\mu\mu'},M)$ be two inverse systems in HTop. A {\it pro-morphism} of inverse
systems, $(f,f_{\mu}): \mathbf{X} \rightarrow \mathbf{Y} $, consists of an index function $f : M \rightarrow\Lambda$ and of mappings $f_{\mu}: X_{f(\mu)} \rightarrow Y_{\mu} $, $\mu\in M$, such that for every related pair $\mu\leq \mu'$ in $M$, there exists a $\lambda\in\Lambda$, $\lambda\geq f(\mu),f(\mu')$ so that,  $$q_{\mu\mu'}f_{\mu'}p_{f(\mu')\lambda}\simeq f_{\mu}p_{f(\mu)\lambda}.$$

The {\it composition} of two pro-morphisms $(f,f_{\mu}): \mathbf{X} \rightarrow \mathbf{Y} $ and $(g,g_{\nu}): \mathbf{Y} \rightarrow \mathbf{Z}=(Z_{\nu},r_{\nu\nu'},N) $ is also a pro-morphism $ (h,h_{\nu})=(g,g_{\nu})(f,f_{\mu}): \mathbf{X} \rightarrow \mathbf{Z} $, where $h=fg$ and $h_{\nu}=g_{\nu}f_{g(\nu)}$. The {\it identity pro-morphism} on $\mathbf{X}$ is pro-morphism $(1_{\Lambda}, 1_{X_{\lambda}}): \mathbf{X} \rightarrow \mathbf{X} $, where $1_{\Lambda}$ is the identity function.
A pro-morphism $(f,f_{\mu}): \mathbf{X} \rightarrow \mathbf{Y} $ is said to be {\it equivalent} to a pro-morphism $(f',f'_{\mu}): \mathbf{X} \rightarrow \mathbf{Y} $, denoted by $(f,f_{\mu})\sim (f',f'_{\mu})$, provided every $\mu\in M$ admits a $\lambda\in\Lambda$ such that $\lambda\geq f(\mu),f'(\mu)$ and
$$f_{\mu}p_{f(\mu)\lambda}\simeq f'_{\mu}p_{f'(\mu)\lambda}.$$

The relation $\sim$ is an equivalence relation. The {\it category} pro-HTop has as objects all inverse systems $\mathbf{X}$ in HTop and as morphisms all equivalence classes $\mathbf{f}=[(f,f_{\mu})]$. The composition of $\mathbf{f}=[(f,f_{\mu})]$ and $\mathbf{g}=[(g,g_{\nu})]$ in pro-HTop is well defined by putting
$$\mathbf{gf}=\mathbf{h}=[(h,h_{\nu})].$$

An HPol-expansion of a topological space $X$ is a morphism $\mathbf{p} :X\rightarrow \mathbf{X}$ of pro-HTop, where $\mathbf{X}$ belong to pro-HPol characterised by the following two properties:\\
(E1) For every $P\in HPol$ and every map $h:X\rightarrow P$ in HTop, there is a $\lambda\in \Lambda$ and a map $f:X_{\lambda}\rightarrow P$ in HPol such that $fp_{\lambda}\simeq h$.\\
(E2) If $f_0, f_1:X_{\lambda}\rightarrow P$ is satisfy $f_0p_{\lambda}\simeq f_1p_{\lambda}$, then there exists a $\lambda'\geq\lambda$ such that $f_0p_{\lambda\lambda'}\simeq f_1p_{\lambda\lambda'}$.

Let $\mathbf{p} :X\rightarrow \mathbf{X}$ and $\mathbf{p'} :X\rightarrow \mathbf{X'}$ be two HPol-expansions of the same space $X$ in HTop, and let $\mathbf{q} : Y \rightarrow \mathbf{Y}$ and $\mathbf{q'} : Y \rightarrow \mathbf{Y'}$ be two HPol-expansions of the same space $Y$ in HTop. Then there exist two natural isomorphisms $\mathbf{i}:\mathbf{X}\rightarrow \mathbf{X}'$ and $\mathbf{j}:\mathbf{Y}\rightarrow \mathbf{Y}'$ in pro-HTop. A morphism $\mathbf{f}:\mathbf{X}\rightarrow \mathbf{Y}$ is said to be {\it equivalent} to a morphism $\mathbf{f'}:\mathbf{X'}\rightarrow \mathbf{Y'}$, denoted by $\mathbf{f}\sim\mathbf{f'}$, provided the following diagram in pro-HTop commutes:
\begin{equation}
\label{dia}\begin{CD}
\mathbf{X}@>\mathbf{i}>>\mathbf{X'}\\
@VV \mathbf{f}V@V \mathbf{f'}VV\\
\mathbf{Y}@>\mathbf{j}>>\mathbf{Y'},
\end{CD}\end{equation}

Now, the {\it shape category} Sh is defined as follows: The objects of Sh are topological spaces. A morphism $F:X\rightarrow Y$ is the equivalence class $<\mathbf{f}>$ of a mapping $\mathbf{f}:\mathbf{X}\rightarrow \mathbf{Y}$ in pro-HTop.

The {\it composition} of $F=<\mathbf{f}>:X\rightarrow Y$ and $G=<\mathbf{g}>:Y\rightarrow Z$ is defined by the representatives, i.e., $GF=<\mathbf{g}\mathbf{f}>:X\rightarrow Z$. The {\it identity shape morphism} on a space $X$, $1_X:X\rightarrow X$, is the equivalence class $<1_\mathbf{X}>$ of the identity morphism $1_\mathbf{X}$ in pro-HTop.

Let $\mathbf{p}:X\rightarrow \mathbf{X}$ and $\mathbf{q}:Y\rightarrow \mathbf{Y}$ be HPol-expansions of $X$ and $Y$, respectively. Then for every morphism $f:X\rightarrow Y$ in HTop, there is a unique morphism $\mathbf{f}:\mathbf{X}\rightarrow \mathbf{Y}$ in pro-HTop such that the following diagram commutes in pro-HTop.
\begin{equation}
\label{dia}\begin{CD}
\mathbf{X}@<<\mathbf{p}<X\\
@VV \mathbf{f}V@V fVV\\
\mathbf{Y}@<<\mathbf{q}<Y.
\end{CD}\end{equation}

If we take other HPol-expansions $\mathbf{p'}:X\rightarrow \mathbf{X'}$ and $\mathbf{q'}:Y\rightarrow \mathbf{Y'}$, we obtain another morphism $\mathbf{f'}:\mathbf{X'}\rightarrow \mathbf{Y'}$ in pro-HTop such that $\mathbf{f'}\mathbf{p'^*}=\mathbf{q'}f$ and so we have $\mathbf{f}\sim\mathbf{f'}$. Hence every morphism $f\in HTop(X,Y)$ yields an equivalence class $<[\mathbf{f}]>$, i.e., a shape morphism $F:X\rightarrow Y$ which is denoted by $\mathcal{S}(f)$. If we put $\mathcal{S}(X)=X$ for every topological space $X$, then we obtain a functor $\mathcal{S}:HTop\rightarrow Sh$, called the {\it shape functor}.

Similarly, we can define the categories pro-HTop$_*$ and Sh$_*$ on pointed topological spaces (see \cite{MS}).
\section{The topological shape homotopy groups}\label{S}
In \cite{CM} Cuchillo-Ibanez et al. introduced a topology on the set of shape morphisms between two arbitrary topological spaces $X$ and $Y$, $Sh(Y,X)$, as follows:

Assume $\mathbf{X}=(X_{\lambda},p_{\lambda\lambda'},\Lambda)$ is an inverse system
in pro-HPol and let $\mathbf{p} : X\rightarrow\mathbf{X}$ be an HPol-expansion of $X$. For every $\lambda\in \Lambda$ and
$F\in Sh(Y,X)$ take $V^F_{\lambda}=\{G\in Sh(Y,X)|\ \  p_{\lambda}\circ F= p_{\lambda}\circ G \ \ \text{as homotopy classes
to}\ \ X_{\lambda}\}$. It is proved that, $\{V^F_{\lambda}|\ \  F\in Sh(Y,X) \ \ \text{and}\ \ \ \lambda\in \Lambda\}$ is a basis for a topology $T_\mathbf{p}$ on $Sh(Y,X)$. Moreover, this topology depends only on $X$ and $Y$. One can see that if $(\mathbf{X},\mathbf{x})=((X_{\lambda},x_{\lambda}),p_{\lambda\lambda'},\Lambda)$ is an inverse system in pro-HPol$_*$ and $\mathbf{p}:(X,x)\rightarrow (\mathbf{X},\mathbf{x})$ is an HPol$_*$-expansion of $(X,x)$. For every $\lambda\in \Lambda$ and $F\in Sh((S^k,*),(X,x))$ take \\$V^F_{\lambda}=\{G\in Sh((S^k,*),(X,x))|\ \  p_{\lambda}\circ F= p_{\lambda}\circ G \ \ \text{as pointed homotopy classes to}\ \ X_{\lambda}\}$. Then the collection $\{V^F_{\lambda}| F\in Sh((S^k,*),(X,x))\ \ \text{and}\ \ \lambda\in \Lambda\}$ is a basis for a topology $T_\mathbf{p}$ on $Sh((S^k,*),(X,x))$.

On the other hand, we can consider the $k$th shape group $\check{\pi}_k(X,x)$, $k\in \Bbb{N}$,  as the set of all shape morphisms $F:(S^k,*)\rightarrow (X,x)$, $ Sh ((S^k,*),(X,x))$ with the following multiplication which makes it a group.
$$ F+G= <\mathbf{f}>+<\mathbf{g}>=<\mathbf{f}+\mathbf{g}>=<[(f_{\lambda})]+[(g_{\lambda})]>=<[(f_{\lambda}+g_{\lambda})]>,$$
where shape morphisms $F$ and $G$ are represented by morphisms $\mathbf{f}=[(f_{\lambda})]$ and $\mathbf{g}=[(g_{\lambda})]:(S^k,*)\rightarrow (\mathbf{X},\mathbf{x})$ in pro-HPol$_*$, respectively (see \cite{M}).

Now, we intend to show that $\check{\pi}_k(X,x)=Sh((S^k,*),(X,x))$, $k\in \Bbb{N}$, with the above topology is a topological group which we then denote it by $\check{\pi}_k^{top}(X,x)$. Note that Moron et al. \cite{Mo1} introduced a norm on the shape groups of a compactum $X$ which the above topology on shape groups of metric compace spaces is the
same as  the topology induced by this norm (see \cite[Proposition 2]{CM}).

\begin{theorem}\label{quasi}
Let $(X,x)$ be a pointed topological space. Then $\check{\pi}_k^{top}(X,x)$ is a topological group for all $k\in \Bbb{N}$.
\end{theorem}
\begin{proof}
First we show that $\phi:\check{\pi}_k^{top}(X,x)\rightarrow \check{\pi}_k^{top}(X,x)$ given by $\phi(F=<[(f,f_{\lambda})]>)=F^{-1}$ is continuous, where $F^{-1}:(S^k,*)\rightarrow(X,x)$ is represented by $\mathbf{f}^{-1}=(f,f^{-1}_{\lambda}):(S^k,*)\rightarrow(\mathbf{X},\mathbf{x})$ and $f^{-1}_{\lambda}:(S^k,*)\rightarrow(X_{\lambda},x_{\lambda})$ is the inverse loop of $f_{\lambda}$. Let $V_{\lambda}^{F^{-1}}$ be an open neighbourhood of $F^{-1}$ in $\check{\pi}_k^{top}(X,x)$. We know that for any $G=<[(g,g_{\lambda})]>\in V_{\lambda}^F$, $p_{\lambda}\circ G=p_{\lambda}\circ F$ as pointed homotopy classes in $X_{\lambda}$ or equivalently $g_{\lambda}\simeq f_{\lambda}$ rel $\{*\}$. Then $g_{\lambda}^{-1}\simeq f_{\lambda}^{-1}$ rel $\{*\}$ and so $\phi(G)\in  V_{\lambda}^{F^{-1}}$. Therefore the map $\phi$ is continuous.

Second, we show that the map $m:\check{\pi}_k^{top}(X,x)\times \check{\pi}_k^{top}(X,x)\rightarrow \check{\pi}_k^{top}(X,x)$ given by $m(F,G)=F+G$ is continuous, where $F+G$ is the shape morphism represented by $\mathbf{f}+\mathbf{g}=(f, f_{\lambda}+g_{\lambda}):(S^k,*)\rightarrow(\mathbf{X},\mathbf{x})$ and $f_{\lambda}+g_{\lambda}$ is the concatenation of paths. Let $V_{\lambda}^{F+G}$ be an open neighbourhood of $F+G$ in $\check{\pi}_k^{top}(X,x)$. For any $(K,H)\in V_{\lambda}^F\times V_{\lambda}^G$, we have $p_{\lambda}\circ (K+H)=(p_{\lambda}\circ K)+ (p_{\lambda}\circ H)=(p_{\lambda}\circ F)+ (p_{\lambda}\circ G)=p_{\lambda}\circ ( F+G)$ as pointed homotopy classes in $X_{\lambda}$. Hence $m(K,H)\in  V_{\lambda}^{F+G}$ and so $m$ is continuous.
\end{proof}
\begin{remark}\label{re}
The following statements hold for the topological group $\check{\pi}_k^{top}(X,x)$ for all $k\in \Bbb{N}$:\\
(i) For every pointed topological space $(X,x)$, the topological group $\check{\pi}_k^{top}(X,x)$ is Tychonoff and in particular it is Hausdorff (see \cite[Corollary 3]{CM}).\\
(ii) If $F:(X,x)\rightarrow (Y,y)$ is a shape morphism, then $F_*:\check{\pi}_k^{top}(X,x)\rightarrow \check{\pi}_k^{top}(Y,y)$ is continuous (see \cite[Corollary 2]{CM}).\\
(iii) If $(X,x)$ and $(Y,y)$ are two pointed topological spaces and $Sh(X,x)=Sh(Y,y)$, then $\check{\pi}_k^{top}(X,x)\cong\check{\pi}_k^{top}(Y,y)$ as topological groups (see \cite[Corollary 2]{CM}).
\end{remark}
\begin{corollary}
For any $k\in \Bbb{N}$, $\check{\pi}_k^{top}(-)$ is a functor from the pointed shape category of spaces to the category of Hausdorff topological groups.
\end{corollary}
\begin{proof}
This follows from Theorem \ref{quasi} and Remark \ref{re}.
\end{proof}

An inverse system $\mathbf{X}=(X_{\lambda},p_{\lambda\lambda'},\Lambda)$ in HPol is said to be 1-movable provided every $\lambda\in \Lambda$ admits $\lambda'\geq\lambda$ such that for any $\lambda''\geq\lambda$, any polyhedron $P$ with dim $P\leq 1$ and any map $h:P\rightarrow X_{\lambda'}$, there is a map $r:P\rightarrow X_{\lambda''}$ such that $p_{\lambda\lambda''}r\simeq p_{\lambda\lambda'}h$. We say that a topological space $X$ is 1-movable  provided it admits an HPol-expansion $\mathbf{p}:X\rightarrow \mathbf{X}$ such that $\mathbf{X}$ is 1-movable (see \cite{MS}).
\begin{corollary}\label{ad1}
Let $(X,x)$ and $(Y,y)$ be two pointed metric continua and $(X,x)$ be 1-movable. If $X$ and $Y$ have the same homotopy type, then $\check{\pi}_k^{top}(X,x)\cong\check{\pi}_k^{top}(Y,y)$ for all $k\in \Bbb{N}$.
\end{corollary}
\begin{proof}
 Since $X$ and $Y$ have the same homotopy type, $Sh(X)=Sh(Y)$. Since $(X,x)$ is 1-movable and $(X,x)$ and $(Y,y)$ are two pointed metric continua, $Sh(X,x)=Sh(Y,y)$ by \cite[Theorem 2.8.10]{MS}. Hence by Remark \ref{re} part (iii) the result holds.
\end{proof}

Let $\mathbf{X}=(X_n, p_{nn+1})$ be an inverse sequence of compact ANR's and maps and let $\mathbf{p}:X\rightarrow \mathbf{X}$ be an inverse limit. The points $x,x'\in X$ are said to be $\mathbf{p}$-joinable if there exist paths $\omega_n$ in $X_n$, $n\in \Bbb{N}$, such that \\
\begin{align}
\omega_n(0) & = x_n=p_n(x)\nonumber\\
\omega_n(1)& = x'_n=p_n(x')\nonumber\\
p_{nn+1}\omega_{n+1}&\simeq\omega_n \hspace{1pt} rel\{0,1\}.\nonumber
\end{align}
Two points $x,y$ of a metric compactum $X$ are called joinable provided they are $\mathbf{p}$-joinable for some inverse limit $\mathbf{p}:X\rightarrow \mathbf{X}$, where $\mathbf{X}$ is an inverse sequence of compact ANR's. A metric compactum space $X$ is said to be joinable if every two points of $X$ be joinable. Note that any pointed 1-movable metric continuum $(X,x)$ is joinable \cite[Theorem 2.8.8]{MS}. For example every pointed continuum $(X,x)$ in $\Bbb{R}^2$ is 1-movable \cite[Example 2.8.3]{MS} and every pointed metric path connected continuum is 1-movable \cite[Remark 2.8.5]{MS}. Moreover, for a metric continuum space with joinable points $x,x'\in X$, we have $\check{\pi}_k(X,x)\cong \check{\pi}_k(X,x')$, for all $k\in\Bbb{N}$ \cite[Theorem 2.8.9]{MS}. We have the following result which is find in \cite{Mo1}, too.
\begin{theorem}\label{ad2}
Let $X$ be a metric continuum space with joinable points $x,x'\in X$. Then $\check{\pi}_k^{top}(X,x)\cong \check{\pi}_k^{top}(X,x')$, for all $k\in\Bbb{N}$.
\end{theorem}
\begin{proof}
Since $X$ is a metric continuum space and $x,x'\in X$ are joinable points, $Sh(X,x)\cong Sh(X,x')$ by \cite[Theorem 2.8.9]{MS}. Hence by Remark \ref{re} part (iii) the result holds.
\end{proof}

Recall that $\Bbb{N}$-compact spaces are those which can be
embedded in $\Bbb{N}^m$ as closed subspaces, for some cardinal number $m$, where $\Bbb{N}$ is the set of natural numbers with discrete topology \cite{N}.

According to \cite{CM}, a cardinal number $m$ is said to be measurable if a set $X$ of cardinal $m$ admits a countable additive function $\mu$ such that $\mu(X)=1$ and $\mu(x)=0$ for every $x\in X$.
A discrete space is $\Bbb{N}$-compact if and only if its cardinal is nonmeasurable (see \cite{GJ}).

As mentioned in \cite{CM}, the question whether every cardinal number is nonmeasurable is known as the problem
of measurability of cardinal numbers. The assumption that all cardinal numbers are
nonmeasurable is consistent with the axioms of set theory; on the other hand it is not known
whether the assumption of the existence of measurable cardinals is also consistent with the
axioms of set theory. See \cite{E}, \cite{GJ} for further information and references about measurability of cardinal numbers.

The following result follows from \cite[Corollary 4]{CM}.
\begin{theorem}\label{n-com}
If the space $X$ has nonmeasurable cardinal, then $\check{\pi}_k^{top}(X,x)$ is $\Bbb{N}$-compact, for all $k\in \Bbb{N}$.
\end{theorem}

Note that if we assume the nonexistence of measurable cardinals, which is
consistent with the usual axioms of set theory, by Theorem \ref{n-com}, we conclude that $\check{\pi}_k^{top}(X,x)$ is $\Bbb{N}$-compact for all $k\in \Bbb{N}$ and for all topological space $X$.

Let $\mathbf{p}:X\rightarrow (X_{\lambda},p_{\lambda\lambda'},\Lambda)$ be an HPol-expansion of $X$. Brazas \cite{Br} gave a topology $T_{inv}$ on $\check{\pi}_k(X)$ for all $k\in \Bbb{N}$. Here we denote it by $\check{\pi}_k^{inv}(X)$. Since $\pi_k^{qtop}(X_{\lambda})$ is discrete for all $k\in \Bbb{N}$, this topology is pro-discrete. He defined the $k$th topological shape group of $X$ as the limit $\check{\pi}^{inv}_k(X)= \displaystyle{\lim_{\leftarrow}\pi_k^{qtop}(X_{\lambda})}$ which is an inverse limit of discrete groups and so it is a Hausdorff topological group  \cite[Remark 2.19]{Br}. In the follow, we show that these two topologies on $\check{\pi}_k(-)$ are equivalent, i.e. $\check{\pi}_k^{inv}(X,x)=\check{\pi}_k^{top}(X,x)$ for every pointed topological space $(X,x)$. It should be mentioned that with the second view at this topology we will obtain some results which seem are not obtained with the topology of inverse limit such as Corollary \ref{ad1}, Theorems \ref{re2} and \ref{n-com}. Note that the topology on $Sh(X,Y)$ in the proof of the above results is effective.

\begin{theorem}\label{Bra}
Let $(X,x)$ be a pointed topological space and $\mathbf{p}:(X,x)\rightarrow (\mathbf{X},\mathbf{x})=((X_{\lambda},x_{\lambda}),p_{\lambda\lambda'},\Lambda)$ be an HPol$_*$-expansion of $(X,x)$. Then for all $k\in \Bbb{N}$,
$\check{\pi}_k^{top}(X,x)\cong\displaystyle{\lim_{\leftarrow}\pi_k^{qtop}(X_{\lambda},x_{\lambda})}$ as topological groups.
\end{theorem}
\begin{proof}
Let $k\in \Bbb{N}$. Since $\mathbf{p}:(X,x)\rightarrow (\mathbf{X},\mathbf{x})=((X_{\lambda},x_{\lambda}),p_{\lambda\lambda'},\Lambda)$ is an HPol$_*$-expansion of $(X,x)$, $Sh((S^k,*),(X,x))=\displaystyle{\lim_{\leftarrow}Sh((S^k,*),(X_{\lambda},x_{\lambda}))}$ in Top,  by \cite[Theorem 2]{CM}. So $\check{\pi}_k^{top}(X,x)\cong\displaystyle{\lim_{\leftarrow}\check{\pi}_k^{top}(X_{\lambda},x_{\lambda})}$ as topological groups. Since $X_{\lambda}$'s are polyhedra,  $\check{\pi}_k^{top}(X_{\lambda},x_{\lambda})=\pi^{qtop}_k(X_{\lambda},x_{\lambda})$ which are discrete. Hence $\check{\pi}_k^{top}(X,x)\cong\displaystyle{\lim_{\leftarrow}\pi_k^{qtop}(X_{\lambda},x_{\lambda})}$ as topological groups.
\end{proof}
\begin{corollary}\label{invlim}
Let  $(X,x)=\displaystyle{\lim_{\leftarrow}(X_i,x_i)}$, where $X_i$'s are compact polyhedra. Then for all $k\in \Bbb{N}$,
\[\check{\pi}_k^{top}(X,x)\cong\displaystyle{\lim_{\leftarrow}\pi_k^{qtop}(X_i,x_i)}.\]
\end{corollary}
\begin{proof}
Sine $X_i$'s are compact, $\prod_{i\in \Bbb{N}}X_i$ is compact by \cite[Theorem 3.2.4]{E} and since $X_i$'s are Hausdorff, $X=\displaystyle{\lim_{\leftarrow}X_i}$ is a closed subspace of $\prod_{i\in \Bbb{N}}X_i$ by \cite[Proposition 2.5.1]{E}. Hence $\prod_{i\in \Bbb{N}}X_i$ is compact by \cite[Theorem 3.1.2]{E}. Therefore the limit $\mathbf{p}:X\rightarrow (X_i, p_{ii+1},\Bbb{N})$ is an HPol-expansion of $X$ by \cite[Remark 1]{FZ} and so the result holds by Theorem \ref{Bra}.
\end{proof}

Ghane et al. \cite{G2} proved that  $\pi_n^{qtop} (X,x)\cong\displaystyle{\lim_{\leftarrow}\pi_n^{qtop}(X_i,x_i)}$ for n-Hawaiian like space $X$ for all $n\geq 2$. An n-Hawaiian like space $X$ is the natural inverse limit, $\displaystyle{\lim_{\leftarrow}(Y_i^{n},y_i^*)}$, where $(Y_i^{n},y_i^*)=\bigvee_{j\leq i}(X_j^{n},x_j^*)$ is the wedge of $X_j^{n}$'s, where $X_j^{n}$'s are $(n-1)$-connected, locally $(n-1)$-connected, semilocally $n$-connected, and compact CW spaces. Therefore using Corollary \ref{invlim}, we can conclude that for all $n\geq 2$, $\check{\pi}_n^{top}(X,x)\cong\pi_n^{qtop} (X,x)$ where $X$ is an n-Hawaiian like space.
\begin{corollary}
Let $\mathbf{p}:(X,x)\rightarrow (\mathbf{X},\mathbf{x})=((X_{\lambda},x_{\lambda}),p_{\lambda\lambda'},\Lambda)$ be an HPol$_*$-expansion of a pointed topological space $(X,x)$, then the following statements hold for all $k\in \Bbb{N}$:
(i) If the cardinal number of $\Lambda$ is $\aleph_0$ and  $\pi_k^{qtop}(X_{\lambda},x_{\lambda})$ is first countable (second countable) for every $\lambda\in \Lambda$, then so is $\check{\pi}_k^{top}(X,x)$.\\
(ii) If $\pi_k^{qtop}(X_{\lambda},x_{\lambda})$ is totally disconnected for every $\lambda\in \Lambda$, then so is $\check{\pi}_k^{top}(X,x)$.
\end{corollary}
\begin{proof}
The results follow from the fact that the product and the subspace topologies preserve the properties of being first countable, second countable and totally disconnected of spaces.
\end{proof}


\section{Main results}
It is claimed by Biss \cite{B0} that $\pi_{1}^{qtop}(X,x)$ commutes with the product. But there is a gap in his proof and   hence this proposition is not true, in general. For example, consider the Hawaiian earring (HE). Fabel \cite{F2} proved that the product $q\times q:\Omega(HE)\times\Omega(HE)\longrightarrow\pi^{qtop}_1(HE)\times\pi^{qtop}_1(HE)$ is not a quotient map and hence the topology of $\pi_1^{qtop}(HE\times HE)$ is strictly finer than $\pi_1^{qtop}(HE)\times \pi_1^{qtop}(HE)$. Moreover, Fabel \cite{F3} showed that for each $n\geq 1$ there exists a compact, path connected, metric space $X$ such that multiplication is discontinuous in $\pi_n^{qtop}(X,x)$. Hence $\pi_n^{qtop}$ does not preserve finite products in general (see also \cite{Bra}). It is well-known that if  the cartesian product of two spaces $X$ and $Y$ admits an Hpol-expansion, which is the cartesian product of Hpol-expansions of these space, then $X\times Y$ is a product in the shape category \cite{M1}. In this case, we show that topological $k$th shape group commutes with finite products, for all $k\in \mathbb{N}$.
\begin{theorem}\label{pro}
If $X$ and $Y$ are shape path connected spaces with Hpol-expansions $\mathbf{p}: X\rightarrow \mathbf{X}$ and $\mathbf{q}:Y\rightarrow \mathbf{Y}$   such that $\mathbf{p}\times \mathbf{q}:X\times Y\rightarrow \mathbf{X}\times\mathbf{Y}$ is an Hpol-expansion. Then $\check{\pi}_{k}^{top}(X\times Y)\cong \check{\pi}_{k}^{top}(X)\times \check{\pi}_{k}^{top}(Y)$, for all $k\in \mathbb{N}$.
\end{theorem}
\begin{proof}
Let $\mathcal{S}(\pi_{X}):X\times Y\rightarrow X$ and $\mathcal{S}(\pi_{Y}):X\times Y\rightarrow Y$ be the induced shape morphisms of canonical projections and assume that $\phi_{X}:\check{\pi}_{k}^{top}(X\times Y)\rightarrow \check{\pi}_{k}^{top}(X)$ and $\phi_{Y}:\check{\pi}_{k}^{top}(X\times Y)\rightarrow \check{\pi}_{k}^{top}(Y)$ be the induced continuous homomorphisms by $\mathcal{S}(\pi_{X})$  and $\mathcal{S}(\pi_{Y})$, respectively. Then the induced homomorphism $\phi: \check{\pi}_{k}^{top}(X\times Y)\rightarrow \check{\pi}_{k}^{top}(X)\times \check{\pi}_{k}^{top}(Y)$ is continuous. Since $X\times Y$ is a product in Sh, we can define a homomorphism $\psi:\check{\pi}_{k}^{top}(X)\times \check{\pi}_{k}^{top}(Y)\rightarrow \check{\pi}_{k}^{top}(X\times Y)$ by $\psi(F,G)=\lfloor F,G\rfloor$ where $\lfloor F,G\rfloor:S^{k}\rightarrow X\times Y$ is a unique shape morphism with $\mathcal{S}(\pi_{X})(\lfloor F,G\rfloor)=F$ and $\mathcal{S}(\pi_{Y})(\lfloor F,G\rfloor)=G$. In fact, if $F=\langle \mathbf{f}=(f,f_{\lambda})\rangle$ and $G=\langle \mathbf{g}=(g,g_{\mu})\rangle$, then $\lfloor F,G\rfloor=\langle \lfloor \mathbf{f},\mathbf{g}\rfloor\rangle$, in which $\lfloor \mathbf{f},\mathbf{g}\rfloor$ is given by $\lfloor f,g\rfloor_{\lambda\mu}=f_{\lambda}\times g_{\mu}:S^{k}\rightarrow X_{\lambda}\times Y_{\mu}$. By the proof of \cite[Theorem 2.4]{Na}, the homomorphism $\psi$ is well define and moreover, $\phi o \psi=id$ and $\psi o \phi=id$.

To complete the proof, it is enough to show that $\psi$ is continuous. Let $\lfloor F,G\rfloor\in V_{\lambda\mu}^{\lfloor F,G\rfloor}$ be a basic element in the topology on $\check{\pi}_{k}^{top}(X\times Y)$. Considering open sets $F\in V_{\lambda}^{F}$ and $G\in V_{\mu}^{G}$, we show that $\psi(V_{\lambda}^{F}\times V_{\mu}^{G})\subseteq V_{\lambda\mu}^{\lfloor F,G\rfloor}$. Let $H\in V_{\lambda}^{F}$ and $K\in V_{\mu}^{G}$, then $p_{\lambda}oH=p_{\lambda}oF$ and $q_{\mu}oK=q_{\mu}oG$. If $H=\langle (h,h_{\lambda})\rangle$ and $K=\langle (k,k_{\mu})\rangle$, then $p_{\lambda}oH:S^{k}\rightarrow X_{\lambda}$ and $q_{\mu}oK:S^{k}\rightarrow Y_{\mu}$ are shape morphisms given by $p_{\lambda}oh_{\lambda}$ and $q_{\mu}ok_{\mu}$, respectively. Therefore $p_{\lambda}oH\times q_{\mu}oK:S^{k}\rightarrow X_{\lambda}\times Y{\mu}$ is a shape morphism given by $(p_{\lambda}oh_{\lambda})\times (q_{\mu}ok_{\mu})$. On the other hand, $p_{\lambda}\times q_{\mu}(\lfloor H,K\rfloor):S^{k}\rightarrow X_{\lambda}\times Y_{\mu}$ is a shape morphism given by $(p_{\lambda}\times q_{\mu})o(h_{\lambda}\times k_{\mu})=(p_{\lambda}oh_{\lambda})\times (q_{\mu}ok_{\mu})$. So we have $p_{\lambda}\times q_{\mu}(\lfloor H,K\rfloor)=p_{\lambda}oH\times q_{\mu}oK$. Hence
\begin{eqnarray}
p_{\lambda}\times q_{\mu}(\lfloor H,K\rfloor)&=&p_{\lambda}oH\times q_{\mu}oK\nonumber\\
&=&p_{\lambda}oF\times q_{\mu}oG\nonumber\\
&=&p_{\lambda}\times q_{\mu}(\lfloor F,G\rfloor)\nonumber
\end{eqnarray}
which follows that $\psi(H,K)=\lfloor H,K\rfloor\in V_{\lambda\mu}^{\lfloor F,G\rfloor}$.
\end{proof}

Recall that if $X\subseteq Y$, then $X$ is a {\it retract} of $Y$ if there exists a map $r:Y\rightarrow X$ such that $r(x)=x$ for all $x\in X$.
\begin{theorem}\label{ret}
Let $X\subseteq Y$, $r:Y\rightarrow X$ be a retraction and $j:X\rightarrow Y$ be the inclusion map. Then $j_{*}:\check{\pi}_{k}^{top}(X,x)\rightarrow \check{\pi}_{k}^{top}(Y,x)$ is a topological embedding.
\end{theorem}
\begin{proof}
Suppose $\mathcal{S}(j):X\rightarrow Y$ and $\mathcal{S}(r):Y\rightarrow X$ are induced shape morphisms by $j$ and $r$, respectively. Then the induced homomorphisms $j_{*}:\check{\pi}_{k}^{top}(X,x)\rightarrow \check{\pi}_{k}^{top}(Y,x)$ and $r_{*}:\check{\pi}_{k}^{top}(Y,x)\rightarrow \check{\pi}_{k}^{top}(X,x)$ of shape morphisms $\mathcal{S}(j)$ and $\mathcal{S}(r)$ are continuous
by Remark \ref{re}. Let $G=im j_{*}$, it is routine to check that $j_{*}$ and $r_{*|G}$ are inverse of each other. Therefore $j_{*}$ is an embedding.
\end{proof}
Let $(X,x)$ be a  topological space and $\mathbf{p}:(X,x)\rightarrow (\mathbf{X},\mathbf{x})=((X_{\lambda},x_{\lambda}),p_{\lambda\lambda'},\Lambda)$ be an Hpol$_*$-expansion of $X$. We know ${p_{\lambda}}_{*}:\pi_{k}^{qtop}(X,x)\rightarrow \pi_{k}^{qtop}(X_{\lambda},x_{\lambda})$ is a continuous homomorphism, for all $\lambda\in \Lambda$. Then the induced homomorphism $\phi:\pi_{k}^{qtop}(X,x)\rightarrow \check{\pi}_{k}^{top}(X,x)$ is continuous, for all $k\in \mathbb{N}$.
\begin{corollary}\label{inj}
If $(X,x)$ is shape injective, then $\pi_{k}^{qtop}(X,x)$ is a Hausdorff topological group, for all $k\in \mathbb{N}$.
\end{corollary}
\begin{proof}
Let $\alpha\in \pi_{k}^{qtop}(X,x)$. By hypothesis, $\phi$ is one to one, so $\lbrace \alpha\rbrace=\phi^{-1}(\lbrace \phi(\alpha)\rbrace)$. Since $\check{\pi}_{k}^{top}(X,x)$ is Hausdorff and $\phi$ is contiuous, we can conclude that  $\lbrace \alpha\rbrace$ is a closed subset of $\pi_{k}^{qtop}(X,x)$.
\end{proof}
\begin{theorem}\label{re2}
Let $(X,x)$ be a pointed topological space. Then for all $k\in \Bbb{N}$,\\
(i) If $(X,x)\in HPol_*$, then $\check{\pi}_k^{top}(X,x)$ is discrete.\\
(ii) If $\mathbf{p}:(X,x)\rightarrow (\mathbf{X},\mathbf{x})=((X_{\lambda},x_{\lambda}),p_{\lambda\lambda'},\Lambda)$ is an HPol$_*$-expansion of $(X,x)$ and $\check{\pi}_k^{top}(X,x)$ is discrete, then $\check{\pi}_k^{top}(X,x)\leq\check{\pi}_k^{top}(X_{\lambda},x_{\lambda})$, for some $\lambda\in \Lambda$.
\end{theorem}
\begin{proof}
(i) This follows from \cite[Corollary 1]{CM}.\\
(ii) Since $\check{\pi}_k^{top}(X,x)$ is a discrete group. Hence $\{E_x\}$ is an open set of identity point of $\check{\pi}_k^{top}(X,x)$. Thus $\{E_x\}=\cup_{\lambda\in \Lambda_0} V_{\lambda}^F$, where $\Lambda_0\subseteq\Lambda$. Consider the induced homomorphism ${p_{\lambda}}_*:\check{\pi}_k^{top}(X,x)\rightarrow \check{\pi}_k^{top}(X_{\lambda},x_{\lambda})$ given by ${p_{\lambda}}_*(F)=p_{\lambda}\circ F$. Let $G\in ker {p_{\lambda}}_*$, i.e. $p_{\lambda}\circ G=E_{x_{\lambda}}=p_{\lambda}\circ E_x$. Thus $G\in V^{E_x}_{\lambda}\subseteq \cup_{\lambda\in \Lambda_0} V_{\lambda}^F=\{E_x\}$ and so $G=E_x$. Therefore ${p_{\lambda}}_*$ is injective, for all $\lambda\in \Lambda_0$ and $k\in\Bbb{N}$.
\end{proof}

Note that if $\mathbf{p}:(X,x)\rightarrow (\mathbf{X},\mathbf{x})=((X_{\lambda},x_{\lambda}),p_{\lambda\lambda'},\Lambda)$ is an HPol$_*$-expansion of a pointed topological space $(X,x)$, then for all $k\in \Bbb{N}$, $\check{\pi}_k^{top}(X,x)=\displaystyle{\lim_{\leftarrow}\pi_k^{qtop}(X_{\lambda},x_{\lambda})}$ by Theorem \ref{Bra}. Therefore  we can conclude that the $k$th homotopy shape group $\check{\pi}_k^{top}(X,x)$, $k\in \Bbb{N}$, is discrete if all but finitely many $X_{\lambda}$ are simply connected. Indeed, $\check{\pi}_k^{top}(X,x)=\displaystyle{\lim_{\leftarrow}\pi_k^{qtop}(X_{\lambda},x_{\lambda})}$ is embedded in $\prod_{\lambda\in \Lambda}\pi_k^{qtop}(X_{\lambda},x_{\lambda})$ and a product of infinitely many discrete spaces having more than one point is not discrete unless all but finitely many of these spaces are singleton. Therefore all but finitely many of $\pi_k^{qtop}(X_{\lambda},x_{\lambda})$ should be trivial.
\begin{example}\label{HE}
Let $(HE,p=(0,0))=\displaystyle{\lim_{\leftarrow}(X_i,p_i)}$ be the Hawaiian Earring where $X_j=\vee_{i=1}^jS^1_i$. The first shape group $\check{\pi}_1^{top}(HE,p)$ is not embedded in $\check{\pi}_1^{top}(X_j,p_j)$, for any $j\in I$. Therefore, by Theorem \ref{re2}, we can conclude that $\check{\pi}_1^{top}(HE,p)$ is not discrete.
\end{example}

In the follow, we intend to obtain a relationship between the shape homotopy groups and topological shape homotopy groups of a topological space $X$.
\begin{remark}
Let $F:(X,x)\rightarrow (Y,y)$ be a shape morphism represented by $\mathbf{f}:(\mathbf{X},\mathbf{x})=((X_{\lambda},x_{\lambda}),p_{\lambda\lambda'},\Lambda)\rightarrow (\mathbf{Y},\mathbf{y})=((Y_{\mu},y_{\mu}),q_{\mu\mu'},M)$. Let $k\in \Bbb{N}$. If the induced morphism $\pi_k(\mathbf{f}):\pi_k(\mathbf{X},\mathbf{x})\rightarrow \pi_k(\mathbf{Y},\mathbf{y})$ is an isomorphism in pro-group, then $\check{\pi}_k^{top}(F):\check{\pi}_k^{top}(X,x)\rightarrow \check{\pi}_k^{top}(Y,y)$ is an isomorphism of topological groups. Indeed using \cite[Theorem 1.1.3]{MS}, for the morphism $\mathbf{f}:\mathbf{X}\rightarrow \mathbf{Y}$ there exist inverse systems $\mathbf{X'}$ and $\mathbf{Y'}$ indexed over the same cofinite directed ordered set $N$ and a morphism $\mathbf{f'}:\mathbf{X'}\rightarrow \mathbf{Y'}$ represented by a level morphism $(1_\nu,f_\nu)$. The isomorphism $\pi_k(\mathbf{f})$ induces an isomorphism $\pi_k(\mathbf{f'})$. Since $X_{\lambda}$'s and $Y_{\mu}$'s are polyhedra, whose topological $k$th homotopy groups are discrete and therefore $\pi_k^{top}(\mathbf{f'})$ is an isomorphism in pro-$\mathcal{C}$, where $\mathcal{C}$ is the category of topological groups. Thus $\check{\pi}_k^{top}(F)$ is an isomorphism by \cite[Proposition 2.5.10]{E}.
\end{remark}

With the above remark it seems that the isomorphism $\check{\pi}_k(X,x)\cong \check{\pi}_k(Y,y)$ as groups implies the isomorphism $\check{\pi}_k^{top}(X,x)\cong \check{\pi}_k^{top}(Y,y)$
as topological groups. But in the following example we exhibit two topological spaces in which whose $k$th shape groups are isomorphic while whose topological $k$th shape groups are not isomorphic, for all $k\in\Bbb{N}$.
\begin{example}
Let $k\in \Bbb{N}$  and Let $\mathbf{X}=(X_n, p_{nn+1}, \Bbb{N})$, where $X_n=\prod_{j=1}^nS^k_j$ is the product of $n$ copies of $k$-sphere $S^k$, for all $n\in \Bbb{N}$ and the bonding morphisms of $\mathbf{X}$ are the projection maps. Put $X=\displaystyle{\lim_{\leftarrow} X_n}$. Since $X_n$'s are metric continua, the inverse limit $X$ is metric continuum by \cite[ Theorem 6.1.20 and Corollary 4.2.5]{E}. Also since bonding morphisms of pro-$\pi_k(X, *)$ are onto, pro-$\pi_k(X, *)$ has the Mittag-Leffler propertyand so $(X, *)$ is 1-movable by \cite[Theorem 2.8.4]{MS}. Thus the shape homotopy groups of $X$ independent of the base point by Theorem \ref{ad2}. Since $X_n$'s are compact Hausdorff, the mapping $P:X\rightarrow \mathbf{X}$ is an Hpol-expansion of $X$ by \cite[Remark 1]{FZ} and therefore $\check{\pi}_k(X)\cong \displaystyle{\lim_{\leftarrow}\pi_k(X_n)}\cong \prod \Bbb{Z}$. Let $Y$ be an Eilenberg-MacLane space $K(\prod \Bbb{Z}, k)$. Since $Y\in HPol$, $\check{\pi}_k(Y)=\pi_k(Y)=\prod \Bbb{Z}$ and this group independent of the base point. On the other hand since $Y\in HPol$, $\check{\pi}_k^{top}(Y)$ is discrete by Theorem \ref{re2} part (i) while $\check{\pi}_k^{top}(X)$ is not discrete. Note that a product of infinitely many discrete spaces having more than one point is not discrete unless all but finitely many of these spaces are singleton. Thus $\check{\pi}_k^{top}(X)$ and $\check{\pi}_k^{top}(Y)$ are not isomorphic and hence $X$ and $Y$ do not have the same shape type and therefore these spaces do not have the same homotopy type.
\end{example}

Note that the $k$th shape groups of spaces $X$ and $Y$ in the above example are isomorphic while whose topological $k$th shape groups are not isomorphic, for all $k\in\Bbb{N}$. Therefore these spaces show that the functor $\check{\pi}_k^{top}$ succeeds in distinguishing the homotopy type of $X$ and $Y$ while $\check{\pi}_k$ fails.

The following result seems interesting.
\begin{theorem}\label{com}
Let $\mathbf{p}:(X,x)\rightarrow ((X_n,x_n),p_{nn+1},\Bbb{N})$ be an HPol$_*$-expansion of $(X,x)$ such that $p_{nn+1}$ is onto for all $n\in \Bbb{N}$. If $\pi_k(X_1,x_1)$ is finite, then $\check{\pi}_k^{top}(X,x)$ is compact, for all $k\in \Bbb{N}$.
\end{theorem}
\begin{proof}
Let $\{V_n^F |\ \  F\in \mathcal{F}\subseteq Sh((S^k,*),(X,x))\ \ \mbox{and}\ \  n\in \Bbb{N}\}$ be an open cover of $Sh((S^k,*),(X,x))$. Since $V_n^F\subseteq V_1^F$ for all $n\in \Bbb{N}$, we have
\begin{equation}\label{compact}
Sh((S^k,*),(X,x))\subseteq \cup _{F\in \mathcal{F}} V_1^F.
\end{equation}
Put $\mathcal{A}=\{ V_1^F|\ \ F\in \mathcal{F}\}$. It is sufficient to show that $\mathcal{A}$ is finite. For this, we define a bijection map $\phi: \mathcal{A}\rightarrow \pi_k(X_1,x_1)$. Let $V_1^F\in \mathcal{A}$, where $F$ is represented by $\mathbf{f}:(S^k,*)\rightarrow (\mathbf{X},\mathbf{x})$ consists of $f_n:(S^k,*)\rightarrow (X_n,x_n)$, $n\in \Bbb{N}$. We define $\phi(V_1^F)=[f_1]$. Let $F,G\in \mathcal{F}$. Then $V_1^F=V_1^G$ if and only if $p_1\circ F=p_1\circ G$ as pointed homotopy classes to $X_1$, by definition of $V_1^F$. $p_1\circ F=p_1\circ G$ as pointed homotopy classes to $X_1$ if and only if $f_1\simeq g_1$ rel $\{*\}$, i.e. $[f_1]=[g_1]$. Hence $\phi$ is well-defined and injective. Also $\phi$ is onto. Indeed, let $[g_1]\in \pi_k(X_1,x_1)$. Since $p_{12}$ is onto, we can define the map $g_2:(S^k,*)\rightarrow (X_2,x_2)$ with $p_{12}g_2=g_1$. Inductively we can introduce maps $g_n:(S^k,*)\rightarrow (X_n,x_n)$ such that $p_{nn+1}g_{n+1}=g_n$ for all $n\in \Bbb{N}$ which induce the map $\mathbf{g}:(S^k,*)\rightarrow (\mathbf{X},\mathbf{x})$ of pro-HTop$_*$. Now, $G=<\mathbf{g}>$ is a shape morphism in $Sh((S^k,*),(X,x))$. By (\ref{compact}) there is an $F\in \mathcal{F}$ such that $G\in V_1^F$. Thus $p_1\circ F=p_1\circ G$ as pointed homotopy classes to $X_1$ or equivalently $[f_1]=[g_1]$. Hence $\phi(V_1^F)=[g_1]$ and so $\phi$ is onto.
\end{proof}
A similar result of the above theorem is in \cite{CM2} as follows:\\
{\it Let $(Y,y_0)$ be a pointed metric compact space such that\\
\[(Y,y_0)=\displaystyle{\lim_{\leftarrow}\{(P_n,\{p_0\}_n),\phi_n\}},\]
where $(P_n,\{p_0\}_n)$ are pointed polyhedra with finite $m$-homotopy group (for all $n\in \Bbb{N}$). Then the $m$th-shape group $\check{\pi}_m(X,x)$ is compact \cite[Corollary 5.4]{CM2}}.


Let $\mathbf{p}:(X,x)\rightarrow (\mathbf{X},\mathbf{x})=((X_{\lambda},x_{\lambda}),p_{\lambda\lambda'},\Lambda)$ be an HPol$_*$-expansion of a pointed topological space $(X,x)$, then in viewpoint of $\check{\pi}_k^{inv}(X,x)=\displaystyle{\lim_{\leftarrow}\pi_k^{qtop}(X_{\lambda},x_{\lambda})}$ \cite{Br} which is equivalent to $\check{\pi}_k^{top}(X,x)$ for all $k\in \Bbb{N}$ by Theorem \ref{Bra}, the $k$th  topological shape group $\check{\pi}_k^{top}(X,x)$ is compact if $\pi_k^{qtop}(X_{\lambda},x_{\lambda})$ is finite, for all $\lambda\in \Lambda$.

Recall that an inverse system $\mathbf{X}=(X_{\lambda},p_{\lambda\lambda'},\Lambda)$ of pro-HTop is said to be movable if every $\lambda\in \Lambda$ admits a $\lambda'\geq \lambda$ such that each
$\lambda''\geq \lambda$ admits a morphism $r:X_{\lambda'}\rightarrow X_{\lambda''}$ of HTop with $p_{\lambda\lambda''}\circ r\simeq p_{\lambda\lambda'}$. We say that a topological space $X$ is movable  provided it admits an HPol-expansion $\mathbf{p}:X\rightarrow \mathbf{X}$ such that $\mathbf{X}$ is a movable inverse system of pro-HPol \cite{MS}.

In the following, we show that movability can be transferred from a pointed topological spaces $(X,x)$ to $\check{\pi}_k^{top}(X,x)$ under some condition. First we need the following results.
\begin{lemma}\label{mov}
If $(\mathbf{X},\mathbf{x})=((X_{\lambda},x_{\lambda}),p_{\lambda\lambda'},\Lambda)$ is a movable (uniformly movable) inverse system, then  $\mathbf{Sh}((S^k,*),(X,x))= (Sh((S^k,*),(X_{\lambda},x_{\lambda})), (p_{\lambda\lambda'})_*,\Lambda)$  is also a movable (uniformly movable) inverse system, for all $k\in \Bbb{N}$.
\end{lemma}
\begin{proof}
Let $\lambda\in \Lambda$. Since $(\mathbf{X},\mathbf{x})$ is a movable inverse system, there is a $\lambda'\geq\lambda$ such that for every $\lambda''\geq\lambda$ there is a map $r:(X_{\lambda'},x_{\lambda'})\rightarrow (X_{\lambda''}, x_{\lambda''})$ such that $p_{\lambda\lambda''}\circ r \simeq p_{\lambda\lambda'}$ rel$\{x_{\lambda'}\}$. We consider $r_*: Sh((S^k,*),(X_{\lambda'},x_{\lambda'}))\rightarrow Sh((S^k,*),(X_{\lambda''},x_{\lambda''}))$. Hence $(p_{\lambda\lambda''})_*\circ r_* \simeq (p_{\lambda\lambda'})_*$ and so $\mathbf{Sh}((S^k,*),(X,x))$ is movable.
\end{proof}
\begin{remark}\label{movre}
Let $(X,x)$ be a movable space. By definition, there exists an HPol$_*$-expansion $\mathbf{p}:(X,x)\rightarrow (\mathbf{X},\mathbf{x})$ such that $(\mathbf{X},\mathbf{x})$ is a movable inverse system. Now, if $\mathbf{p_*}:Sh((S^k,*),(X,x))\rightarrow \mathbf{Sh}((S^k,*),(X,x))$ is an HPol$_*$-expansion, then using Lemma \ref{mov}, we can conclude that $\check{\pi}_k^{top}(X,x)$ is a movable topological group for all $k\in \Bbb{N}$.
\end{remark}
\begin{remark}
By \cite[Theorem 2]{CM} if $\mathbf{p}:(X,x)\rightarrow (\mathbf{X},\mathbf{x})$ is an HPol$_*$-expansion of $X$, then $\mathbf{p_*}:Sh((S^k,*),(X,x))\rightarrow \mathbf{Sh}((S^k,*),(X,x))$ is an inverse limit of $\mathbf{Sh}((S^k,*),(X,x))= (Sh((S^k,*),(X_{\lambda},x_{\lambda})), (p_{\lambda\lambda'})_*,\Lambda)$. Now, if $Sh((S^k,*),(X_{\lambda},x_{\lambda}))$ is a compact polyhedron for all $\lambda\in \Lambda$, then by \cite[Remark 1]{FZ} $\mathbf{p_*}$ is an HPol$_*$-expansion of $Sh((S^k,*),(X,x))$.
\end{remark}

Now, we intend to find some conditions on a topological space $X$ under which the map $\mathbf{p_*}$ is an HPol$_*$-expansion.
\begin{lemma}\label{explem}
Let $(X,x)$ have an HPol$_*$-expansion $\mathbf{p}:(X,x)\rightarrow ((X_{\lambda},x_{\lambda}),p_{\lambda\lambda'},\Lambda)$ such that $\pi_k(X_{\lambda},x_{\lambda})$ is finite, for every $\lambda\in \Lambda$. Then $\mathbf{p_*}:Sh((S^k,*),(X,x))\rightarrow \mathbf{Sh}((S^k,*),(X,x))$ is an HPol$_*$-expansion of $Sh((S^k,*),(X,x))$, for all $k\in \Bbb{N}$.
\end{lemma}
\begin{proof}
Since $\mathbf{p}:(X,x)\rightarrow ((X_{\lambda},x_{\lambda}),p_{\lambda\lambda'},\Lambda)$ is an HPol$_*$-expansion, $\mathbf{p_*}:Sh((S^k,*),(X,x))\rightarrow \mathbf{Sh}((S^k,*),(X,x))$ is an inverse limit of $\mathbf{Sh}((S^k,*),(X,x))$ by \cite[Theorem 2]{CM}. Since $X_{\lambda}$'s are polyhedra, $\check{\pi}_k(X_{\lambda},x_{\lambda})=\pi_k(X_{\lambda},x_{\lambda})$ which is finite for all $\lambda\in\Lambda$ by the hypothesis. On the other hand $\check{\pi}_k^{top}(X_{\lambda},x_{\lambda})=\pi_k^{qtop}(X_{\lambda},x_{\lambda})$ which is discrete for all $\lambda\in \Lambda$. Hence $Sh((S^k,*),(X_{\lambda},x_{\lambda}))=\check{\pi}_k^{qtop}(X_{\lambda},x_{\lambda})$ is a compact polyhedron. Therefore $\mathbf{p_*}$ is an HPol$_*$-expansion of $Sh((S^k,*),(X,x))$ by \cite[Remark 1]{FZ}.
\end{proof}
\begin{example}(\cite{MS})
Let $\Bbb{R}P^2$ be the real projective plane. Consider the map $\bar{f}:\Bbb{R}P^2\rightarrow \Bbb{R}P^2$ induced by the following commutative diagram:
\begin{equation}
\label{dia2}\begin{CD}
D^2@<<f< D^2\\
@VV \phi V@V \phi VV\\
\Bbb{R}P^2@<<\bar{f}<\Bbb{R}P^2,
\end{CD}
\end{equation}
 where $D^2=\{z\in \Bbb{C}\ \ | \ \ |z|\leq 1\}$ is the unit 2-cell, $f(z)=z^3$ and $\phi:D^2\rightarrow \Bbb{R}P^2$ is the quotient map which identify pairs of points $\{z,-z\}$ of $S^1$. \\
 We consider $X$ as the inverse sequence $$\Bbb{R}P^2\stackrel{\bar{f}}{\longleftarrow}\Bbb{R}P^2\stackrel{\bar{f}}{\longleftarrow}\cdots.$$
 Since $\Bbb{R}P^2 $ is a compact polyhedron, $\mathbf{p}:X\rightarrow (\Bbb{R}P^2, \bar{f}, \Bbb{N})$ is an HPol-expansion  of $X$, by \cite[Remark 1]{FZ}. Since $\pi_k(\Bbb{R}P^2 )\cong \Bbb{Z}_2$ is finite, $\mathbf{p_*}:Sh((S^k,*),(X,x))\rightarrow \mathbf{Sh}((S^k,*),(X,x))$ is an HPol$_*$-expansion of $Sh((S^k,*),(X,x))$, for all $k\in \Bbb{N}$ by Lemma \ref{explem}.
\end{example}
The following corollary is a consequence of Remark \ref{movre} and Lemma \ref{explem}.
\begin{corollary}\label{12}
Let movable space $(X,x)$ have an HPol$_*$-expansion $\mathbf{p}:(X,x)\rightarrow ((X_{\lambda},x_{\lambda}),p_{\lambda\lambda'},\Lambda)$ such that $\pi_k(X_{\lambda},x_{\lambda})$ is finite, for every $\lambda\in \Lambda$. Then $\check{\pi}_k^{top}(X,x)$ is a movable topological group for all $k\in \Bbb{N}$.
\end{corollary}
Note that with the assumptions of Corollary \ref{12} uniformly movability of a space can be transferred to its topological shape group.\\

Suppose that $(X,x)$ is a pointed topological space. We know that $\pi_{k}(X,x)$ has the quotient topology induced by the natural quotient map $q:\Omega^{k}(X,x)\rightarrow \pi_{k}(X,x)$, where $\Omega^{k}(X,x)$ is the $k$th loop space of $(X,x)$ with the compact-open topology. Now we intend to show that $\check{\pi}_{k}^{top}(X,x)$ can be considered as a quotient space of the shape loop space $\check{\Omega}_{k}(X,x)$ which is defined as follows:

Let $\mathbf{p}:(X,x)\rightarrow ((X_{\lambda},x_{\lambda}),p_{\lambda\lambda'},\Lambda)$ be an  HPol$_*$-expansion of $(X,x)$. We define a $k$-{\it shape loop} in $X$ as an element $(a_{\lambda})\in \prod_{\lambda\in \Lambda} \Omega^{k}(X_{\lambda},x_{\lambda})$ such that $p_{\lambda\lambda^{\prime}}(a_{\lambda^{\prime}})\simeq a_{\lambda}$, for all $\lambda^{\prime}\geq\lambda$. The set of all $k$-shape loops in $X$ is called the   {\it $k$th shape loop space} and it is denoted by $\check{\Omega}_{k}(X,x)$. Then $\check{\Omega}_{k}(X,x)$ is a topological space as subspace of $\prod_{\lambda\in \Lambda} \Omega^{k}(X_{\lambda},x_{\lambda})$.
We define a natural map $q:\check{\Omega}_{k}(X,x)\rightarrow \check{\pi}_{k}^{top}(X,x)$ by $q((a_{\lambda}))=\langle [(a_{\lambda})]\rangle$, for all $(a_{\lambda})\in \prod_{\lambda\in \Lambda} \Omega^{k}(X_{\lambda},x_{\lambda})$.

 A map $f:X\rightarrow Y$ is called a {\it bi-quotient map} if, whenever $y\in Y$ and $\mathcal{U}$ is a covering of $f^{-1}(y)$ by open subsets of $X$, finitely many $f(U)$, where $U\in \mathcal{U}$, cover some neighbourhood of $y$ in $Y$ \citep[Definition 1.1]{ME}. Indeed, if $Y$ is discrete, then any map $f:X\rightarrow Y$ is a bi-quotient map. Moreover, we know that any product of bi-quotient maps is a bi-quotient map \cite[Theorem 1.2]{ME}.
\begin{theorem}\label{quo}
The map $q:\check{\Omega}_{k}(X,x)\rightarrow \check{\pi}_{k}^{top}(X,x)$ is a quotient map, for all $k\in \mathbb{N}$. In other words, the topology of $\check{\pi}_{k}^{top}(X,x)$ agrees with the quotient topology on $\check{\pi}_{k}(X,x)$ induced by $q$.
\end{theorem}

\begin{proof}
It is obvious that $q$ is a surjection. For continuity, suppose $(a_{\lambda})\in \check{\Omega}_{k}(X,x)$, $F=\langle [(a_{\lambda})]\rangle$ and $V_{\lambda_{0}}^{F}$ is an open neighbourhood  of $F$ in $\check{\pi}_{k}^{top}(X,x)$. Since $q_{\lambda_{0}}:\Omega^{k}(X_{\lambda_{0}},x_{\lambda_{0}})\rightarrow \pi_{k}^{qtop}(X_{\lambda_{0}},x_{\lambda_{0}})$ is continuous and $\pi_{k}^{qtop}(X_{\lambda_{0}},x_{\lambda_{0}})$ is discrete, $[a_{\lambda_{0}}]=q_{\lambda_{0}}^{-1}(\lbrace [a_{\lambda_{0}}]\rbrace)$ is an open subset of $\Omega^{k}(X_{\lambda_{0}},x_{\lambda_{0}})$. Then $V=\prod_{\lambda\neq \lambda_{0}} \Omega^{k}(X_{\lambda},x_{\lambda})\times [a_{\lambda_{0}}]$ is an open subset of $\prod_{\lambda} \Omega^{k}(X_{\lambda},x_{\lambda})$ and hence $U=V\cap \check{\Omega}_{k}(X,x)$ is an open subset of $\check{\Omega}_{k}(X,x)$. We show that $q(U)\subseteq V_{\lambda_{0}}^{F}$. For each $u=(b_{\lambda})$ in $U$, we have $p_{\lambda\lambda^{\prime}}(b_{\lambda^{\prime}})\simeq b_{\lambda}$ for all $\lambda^{\prime}\geq \lambda$ and $b_{\lambda_{0}}\simeq a_{\lambda_{0}}$. Hence $q(u)=\langle [(b_{\lambda})]\rangle\in V_{\lambda_{0}}^{F}$.

Now we show that $q$ is a quotient map. Since $\pi_{k}^{qtop}(X_{\lambda},x_{\lambda})$ is discrete, then $q_{\lambda}:\Omega^{k}(X_{\lambda},x_{\lambda})\rightarrow \pi_{k}^{qtop}(X_{\lambda},x_{\lambda})$ is a bi-quotient map, for all $\lambda\in \Lambda$. Therefore $\phi=\prod q_{\lambda}:\prod_{\lambda\in \Lambda}\Omega^{k}(X_{\lambda},x_{\lambda})\rightarrow \prod_{\lambda\in \Lambda}\pi_{k}^{qtop}(X_{\lambda},x_{\lambda})$ is a bi-quotient and so a quotient map. On the other hand, $\pi_{k}^{qtop}(X_{\lambda},x_{\lambda})$ is discrete and so Hausdorff, for all $\lambda\in \Lambda$. Then $\check{\pi}_{k}^{top}(X,x)=\displaystyle{\lim_{\leftarrow}\pi_k^{qtop}(X_{\lambda},x_{\lambda})}$ is a closed subset of $\prod_{\lambda\in \Lambda}\pi_{k}^{qtop}(X_{\lambda},x_{\lambda})$ by \cite[Theorem 2.5.1]{E}. Therefore, $q=\prod p_{\lambda}|_{\check{\Omega}_{k}(X,x)}:\check{\Omega}_{k}(X,x)\rightarrow \check{\pi}_{k}^{top}(X,x)$ is a quotient map by \cite{Ro}.

\end{proof}

\section*{References}

\bibliography{mybibfile}

\begin{thebibliography}{9999}

\bibitem{A}{}{\sc A. Arhangelskii, M. Tkachenko}, {\it Topological Groups and Related Structures}, Atlantis Studies in Mathematics, 2008.

\bibitem{B0}{}{\sc  D. Biss},  The topological fundamental group and generalized covering spaces, {\it Topology and its Applications}, 124 (2002)  355--371.

\bibitem{Br}{}{\sc J. Brazas}, The topological fundamental group and free topological groups, {\it Topology and its Applications}, 158 (2011) 779--802.

\bibitem{Bra}{}{\sc J. Brazas}, The fundamental group as topological group, {\it Topology and its Applications}, 160 (2013)  170--188.

\bibitem{CM1}{}{\sc J.S. Calcut, J.D. McCarthy},  Discreteness and homogeneity of the topological fundamental group, {\it Topology Proceedings}, 34 (2009) 339--349.

\bibitem{CM2}{}{\sc E. Cuchillo-Ibanez, M.A Moron, F.R. Ruiz del Portal}, Ultrametric spaces, valued and semivalued groups arising from the theorry of shape, preprint.

\bibitem{CM}{}{\sc E. Cuchillo-Ibanez, M.A. Moron, F.R. Ruiz del Portal, J.M.R. Sanjurjo}, A topology for the sets of shape morphisms, {\it Topology Appl}. 94 (1999).

\bibitem{E}{}{\sc R. Engelking}, {\it General Topology}, Monograe Matematyczne 60 (PWN, Warsaw, 1977).

\bibitem{F2}{}{\sc P. Fabel}, Multiplication is discontinuous in the Hawaiian earring group (with the quotient topology), {\it Bull. Polish Acad. Sci. Math.} 59 (2011) 77-83.

\bibitem{F3}{}{\sc P. Fabel}, Compactly generated quasi-topological homotopy groups with discontinuous mutiplication, {\it Topology Proceedings}, 40 (2012) 303-309.

\bibitem{FZ}{}{\sc H. Fischer and A. Zastrow}, The fundamental groups of subsets of closed surfaces inject into their first shape groups, {\it Algebraic and Geometric Topology}, 5 (2005) 1655-1676.

\bibitem{Ro}{}{\sc R. Geoghegan}, Topological Methods in Group Theory, GTM, 243. Springer, New York, 2008.

\bibitem{G1}{}{\sc H. Ghane, Z. Hamed, B. Mashayekhy, H. Mirebrahimi}, Topological homotopy groups, {\it Bull. Belg. Math. Soc. Simon Stevin}, 15:3 (2008) 455-464.

\bibitem{G2}{}{\sc H. Ghane, Z. Hamed, B. Mashayekhy, H. Mirebrahimi}, On topological homotopy groups of n-Hawaiian like spaces, {\it Topology Proceedings}, 36 (2010), 255-266.

\bibitem{GJ}{}{\sc  L. Gillman and M. Jerison}, {\it Rings of Continuous Functions}, Springer, Berlin, 1960.

\bibitem{M1}{}{\sc S. Mardesic}, Strong expansions of products and products in strong shape, {\it Topology Appl}. 140(2004), 81-110.

\bibitem{MS}{}{\sc S. Marde\v{s}i\'{c}, J. Segal}, Shape Theory, North-Holland, Amsterdam, 1982.


\bibitem{ME}{}{\sc E. Michael}, Bi-quotient maps and Cartesian products of quotient maps, Ann. Inst. Fourier (Grenoble) 18 (1968), no. 2, 287-302 (1969).

\bibitem{Mo1}{}{\sc M.A. Moron, F.R. Ruiz del Portal}, Shape as a Cantor completion process, {\it Math. Z.} 225 (1997),
67–-86.

\bibitem{Mo2}{}{\sc M.A. Moron, F.R. Ruiz del Portal}, Ultrametrics and infinite dimensional Whitehead theorems in
shape theory, {\it Manuscripta Math.} 89 (1996), 325-–333.

\bibitem{M}{}{\sc  M. Moszy\'{n}sk\'{a}}, Various approaches to the fundamental groups, {\it Fund. M ath.} 78 (1973) 107-–118.

\bibitem{Na}{}{\sc  T.Nasri, B.Mashayekhy, H. Mirebrahimi}, On products in the coarse shape category, to appear in { \it Glasnik Matematicki}.\\
\bibitem{N}{}{\sc P. Nyikos}, Not every 0-dimensional real compact space is $\Bbb{N}$-compact, {\it Bulletin of the American Mathematical Society}, Volume 77, Number 3, May 1971

\bibitem{P1}{}{\sc A. Pakdaman, H. Torabi, B. Mashayekhy}, On H-groups and their applications to topological fundamental groups, preprint,  arXiv:1009.5176v1.








\end{thebibliography}

\end{document}